\documentclass[11pt]{article}
\usepackage[spanish,english]{babel}
\usepackage{amsmath}
\usepackage{amsfonts}
\usepackage{amsthm}
\usepackage{epsfig, subfigure}
\usepackage{amscd,latexsym,amssymb}
\usepackage{graphicx}
\usepackage{url}
\usepackage{color}
\usepackage{enumerate}
\definecolor{gray}{gray}{0.5}

\definecolor{morat}{rgb}{0.6, 0.2, 0.5}
\definecolor{verd}{rgb}{0.2, 0.8, 0.2}

\newcommand{\bc}{\begin{center}}
\newcommand{\ec}{\end{center}}

\newtheorem{thm}{Theorem}
\newtheorem{prop}{Proposition}
\newtheorem{cor}{Corollary}
\newtheorem{defi}{Definition}
\newtheorem{lem}{Lemma}

\begin{document}

\selectlanguage{english}


\title{
 On the Partition Dimension and the Twin Number of  a Graph
 \thanks{ Research partially supported by grants MINECO MTM2015-63791-R, Gen. Cat. DGR 2014SGR46 and MTM2014-60127-P. }
}

\author{
C. Hernando\thanks{Universitat Polit\`ecnica de Catalunya, Barcelona, Spain, carmen.hernando@upc.edu}
\and M. Mora\thanks{Universitat Polit\`ecnica de Catalunya, Barcelona, Spain, merce.mora@upc.edu}
\and I. M. Pelayo\thanks{{\it Corresponding author}:
Departament de Matem{\`a}tiques, ESAB, Avinguda del Canal
Ol¡mpic, s/n 08860 Castelldefels, Spain.ignacio.m.pelayo@upc.edu}
}

\maketitle

\begin{abstract}
A  partition $\Pi$ of the vertex set of a connected graph $G$ is  a \emph{locating partition} of $G$ if every vertex is uniquely determined by its vector of distances to the elements of $\Pi$.  
The \emph{partition  dimension}  of $G$ is the minimum  cardinality of a locating partition of $G$. 
A pair of vertices $u,v$ of a graph $G$ are  called \emph{twins} if they have exactly the same set of neighbors other than $u$ and $v$.
A \emph{twin class} is a maximal set of pairwise twin vertices.
The \emph{twin number} of a graph $G$ is the maximum cardinality of a twin class of $G$.

In this paper we undertake the study of the partition dimension of a graph by also considering its twin number.
This approach allows us to obtain the set of connected graphs of order $n\ge9$ having partition dimension $n-2$.
This set is formed by exactly 15 graphs, instead of 23, as was wrongly stated in the paper ''Discrepancies between metric dimension and partition dimension of a connected graph''  (Disc. Math. 308 (2008) 5026--5031).

\vspace{.1cm}\noindent {\it Key words:}  locating set; locating partition;  metric dimension; partition dimension; twin number

\end{abstract}

\section{Introduction}

All the graphs considered are undirected, simple, finite and \underline{connected}. 
The vertex set and edge set of a graph $G$ are denoted by $V(G)$ and $E(G)$. 
Let $v$ be a vertex of  $G$. 
The \emph{open neighborhood} of $v$ is $\displaystyle N_G(v)=\{w \in V:vw \in E\}$, 
and the \emph{closed neighborhood} of $v$ is $N_G[v]=N(v)\cup \{v\}$.
The \emph{degree} of $v$ is $\deg_G(v)=|N_G(v)|$.
If $N_G[v]=V(G)$ (resp. $\deg_G(v)=1$), then $v$ is called \emph{universal} (resp. a \emph{leaf}).
Let $W$ be a subset of vertices of a graph $G$.
The open neighborhood of $W$ is $\displaystyle N_G(W)=\cup_{v\in W} N_G(v)$, and the closed neighborhood of $W$ is $N_G[W]=N_G(W)\cup W$.
The subgraph of $G$ induced by $W$, denoted by $G[W]$, has as vertex set $W$ and $E(G[W]) = \{vw \in E(G) : v \in W,w \in W\}$.
The \emph{complement} of $G$,  denoted by $\overline{G}$, is the  graph on the same vertices as $G$ such that two vertices are adjacent in $\overline{G}$ if and only if they are not adjacent in $G$.
Let $G_1$, $G_2$ be two graphs having disjoint vertex sets. 
The (disjoint) \emph{union} $G=G_1+G_2$ is the graph such that  $V(G)=V(G_1)\cup V(G_2)$ and $E(G)=E(G_1)\cup E(G_2)$. 
The \emph{join} $G=G_1\vee G_2$ is the graph such that 
$V(G)=V(G_1)\cup V(G_2)$ and $E(G)=E(G_1)\cup E(G_2)\cup \{uv:u\in V(G_1),v\in V(G_2)\} $.

The distance between vertices $v,w\in V(G)$ is denoted by $d_G(v,w)$, or $d(v,w)$ if the graph $G$ is clear from the context. 
The diameter of $G$ is ${\rm diam}(G) = \max\{d(v,w) : v,w \in V(G)\}$.
The distance between a vertex $v\in V(G)$ and a set of vertices $S\subseteq V(G)$, denoted by  $d(v,S)$ is the minimum of the distances between $v$ and the vertices of $S$, that is to say, $d(v,S)=\min\{d(v,w):w\in S\}$.
Undefined terminology can be found in \cite{chlezh11}.

A vertex $x\in V(G)$ \emph{resolves} a pair of vertices $v,w\in V(G)$ if $d(v,x)\ne d(w,x)$. 
A set of vertices $S\subseteq V(G)$ is a \emph{locating set} of $G$, 
if every pair of distinct vertices of $G$ are resolved by some vertex in $S$. 
The \emph{metric dimension} $\beta(G)$ of $G$ is the minimum cardinality of a locating set. 
Locating sets  were first defined by \cite{hararymelter} and \cite{slater}, and they have since been widely investigated (see \cite{chmppsw07,hmpsw10} and their references).

Let $G=(V,E)$ be a graph of order $n$. 
If $\Pi=\{S_1,\ldots,S_k\}$ is a partition of $V$, we denote by $r(u|\Pi)$ 
the vector of distances between a vertex $u\in V$ and the elements of  $\Pi$, that is, $r(u,\Pi)=(d(u,S_1),\dots ,d(u,S_k))$.  
The partition $\Pi$ is called a  \emph{locating partition} of $G$ if, for any pair of distinct vertices $u,v\in V$, 
$r(u,\Pi)\neq r(v,\Pi)$.
Observe that to  prove that a given partition is locating, it is enough to check that the vectors of distances of every pair of vertices belonging to the same part are different.
The \emph{partition  dimension} $\beta_p(G)$ of $G$ is the minimum  cardinality of a locating partition of $G$. 
Locating partitions were introduced in \cite{ChaSaZh00}, and further studied in  
\cite{bada12,chagiha08,fegooe06,ferogo14,gyro10,gyjakuta14,grstramiwi14,royeku16,royele14,tom08,toim09,tojasl07}.
Next, some known results involving the partition dimension are shown.

\begin{thm} [\cite{ChaSaZh00}]\label{mdpd}
Let $G$  be a graph of order $n\ge3$ and diameter ${\rm diam}(G)=d$
\begin{enumerate}
\item $\beta_p(G) \le \beta(G)+1$.
\item $\beta_p(G) \le n-d+1$. Moreover, this bound is sharp.
\item $\beta_p(G)=n-1$ if and only if $G$ is isomorphic to either the star $K_{1,n-1}$, or 
	the complete split graph $K_{n-2} \vee \overline{K_2}$,  or the graph $K_1\vee (K_1+K_{n-2})$.
\end{enumerate}
\end{thm}

In \cite{tom08}, its author approached the characterization of the set of graphs of order $n\ge9$ having partition dimension $n-2$, presenting a collection of 23 graphs (as a matter fact there are 22, as the so-called graphs
 $G_4$ and $G_6$ are isomorphic).
Although employing a different notation (see Table \ref{tab.equivalencias}), the characterization given in this paper is the following.

\begin{thm} [\cite{tom08}]\label{n-2 wrong}
Let $G=(V,E)$ be a  graph of order $n\ge9$. Then  $\beta_p(G)=n-2$ if and only if it belongs either to the family $\{H_i\}_{i=1}^{15}$, except $H_7$, (see Figure \ref{taun234}) or to the family $\{F_i\}_{i=1}^{8}$ (see Figure \ref{pdn-2 wrong}).
\end{thm}

\begin{figure}[!hbt]
\begin{center}
\includegraphics[width=0.85\textwidth]{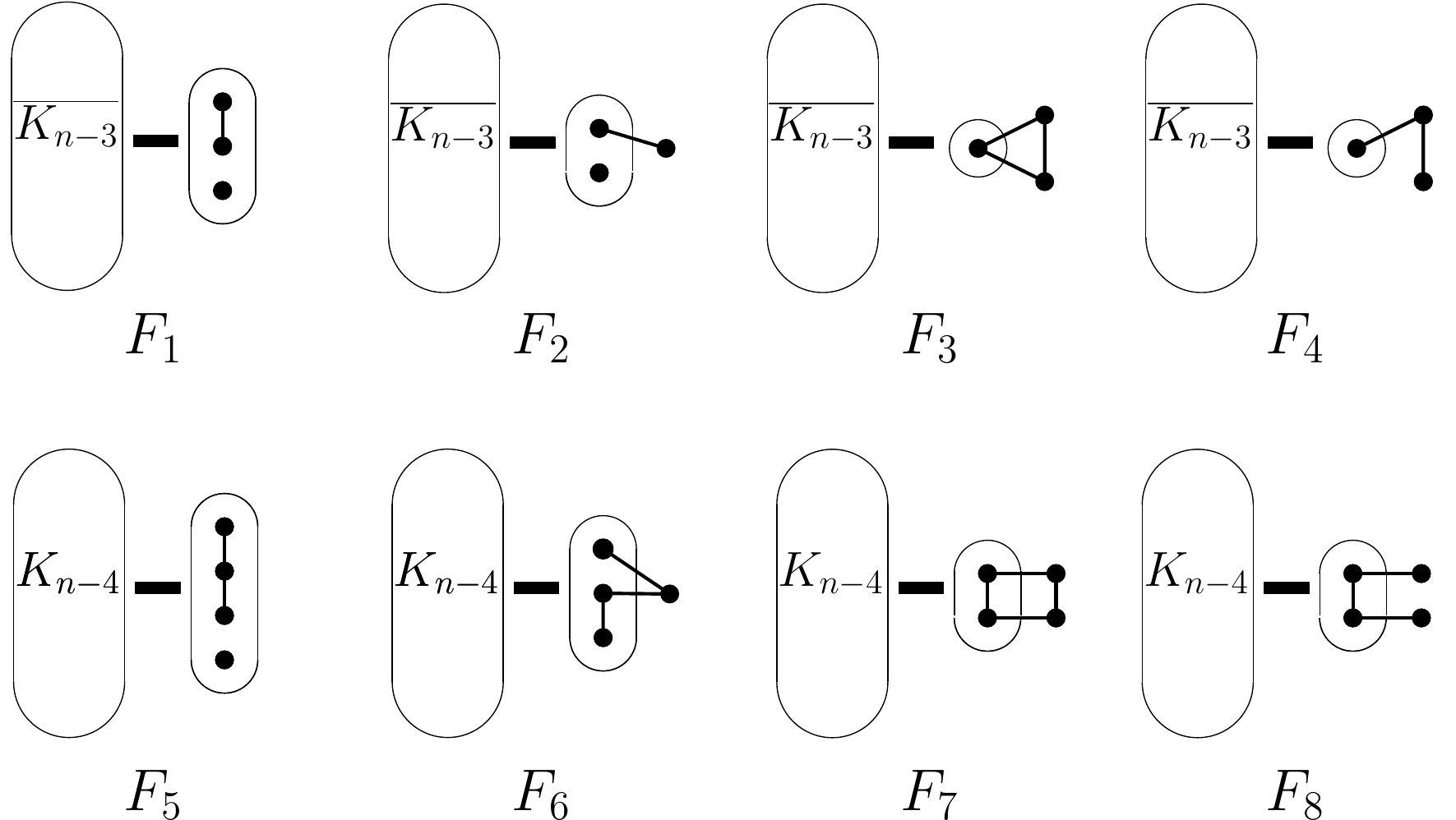}
\caption{ The thick horizontal segment means the join operation $\vee$. 
For example: $F_1\cong \overline{K_{n-3}}\vee (K_2+K_1)$, 
$F_3\cong K_1 \vee(\overline{K_{n-3}}+K_2)$ and 
$F_5\cong \overline{K_{n-4}}\vee (P_3+K_1)$.}\label{pdn-2 wrong}
\end{center}
\end{figure}

\begin{table}[h]
 \begin{center}
 \begin{tabular}{|c|cccccccc|}
  \hline
    Figure \ref{pdn-2 wrong} & $F_1$ & $F_2$ & $F_3$ & $F_4$ & $F_5$  & $F_6$   & $F_7$   & $F_8$ \\
  \hline
    Paper \cite{tom08} & $G_5$ & $K_{2,n-2}-e$ & $K_{1,n-1}+e$ & $G_{11}$ & $K_n-E(K_{1,3}+e)$  & $G_3$   & $G_7$   & $G_{12}$ \\
  \hline
 \end{tabular}
 \end{center}
 \caption{The second row contains the names used in \cite{tom08} for  the graphs shown in Figure\,\ref{pdn-2 wrong}.}
 \label{tab.equivalencias}
\end{table}

Thus, in particular, and according to \cite{tom08}, for every $n\ge9$,  all the graphs $G$ displayed in Figure \ref{pdn-2 wrong} satisfy $\beta_p(G)=n-2$. However, for all of them, it holds that $\beta_p(G)=n-3$, as we will prove in this paper (see Corollaries \ref{f1234}, \ref{f56} and \ref{f78}).
As a matter of example, we show next that $\beta_p(F_1)\le n-3$, for every $n\ge7$.
First, notice that $F_1\cong \overline{K_{n-3}}\vee (K_2+K_1)$.
Next, if $V(\overline{K_{n-3}})=\{v_1, \dots, v_{n-3}\}$, $V(K_2)=\{v_{n-2},v_{n-1}\}$ and  $V(K_1)=\{v_n\}$, then consider the partition
$\Pi=\{ \{v_1,v_{n-2}\}, \{v_2,v_{n-1}\}, \{v_3,v_{n}\}, \{v_4\}, \dots, \{v_{n-3} \} \}$.
Finally, observe that $\Pi$ is a locating partition of $F_1$ since $2=d(v_i,v_{4})\neq d(v_{n+i-3},v_4)=1$, for every $i\in\{1,2,3\}$.

\noindent The main contribution of this work is, after showing that the theorem of characterization presented in \cite{tom08} is far for being true, finding the correct answer to this problem. 
Motivated by this objective, we introduce the so-called \emph{twin number} $\tau(G)$ of a connected graph $G$, and present a list of basic properties, some of them directly related to the partition dimension $\beta_p(G)$.

\noindent The rest of the paper is organized as follows. 
Section 2 is devoted to introduce the notions of twin class and twin number, and to show some basic properties.
In Section 3, subdivided in three subsections, a number of results involving both the twin number and the partition dimension of a graph are obtained. Finally, Section 4 includes a theorem of characterization presenting, for every $n\ge9$, which graphs $G$ of order $n$ satisfy $\beta_p(G)=n-2$.

\section{Twin number}\label{sec.tn}

A pair of vertices $u,v\in V$ of a graph $G=(V,E)$ are  called \emph{twins} if they have exactly the same set of neighbors other than $u$ and $v$. 
A \emph{twin set} of $G$ is any set of pairwise twin vertices of $G$.
If $uv\in E$, then they are called \emph{true twins}, and otherwise \emph{false twins}.
It is easy to verify that the so-called \emph{twin relation} is an equivalence relation on $V$, and that
every equivalence class is either a clique or a stable set.
An equivalence class of the twin relation is referred to as a \emph{twin class}.

\begin{defi}
{\rm The \emph{twin number} of a graph $G$, denoted by $\tau (G)$, is the maximum cardinality of a twin class of $G$.
Every twin set of cardinality $\tau (G)$ will be referred to as a \emph{$\tau$-set}. }
\end{defi}

As a direct consequence of these definitions, the following list of  properties hold.

\begin{prop}\label{twin.list}
 Let $G=(V,E)$ be a graph of order $n$. Let $W$ be a twin set of $G$. Then

 \begin{enumerate} 

 \item[(1)] If $w_1,w_2\in W$, then $d(w_1,z)=d(w_2,z)$, for every vertex $z\in V\setminus \{w_1,w_2\}$.
 
 \item[(2)] No two vertices of $W$ can belong to the same part of any locating partition.

 \item[(3)] $W$ induces either a complete graph or an empty graph. 

 \item[(4)] Every vertex not in $W$ is either adjacent to all the vertices of $W$ or non-adjacent to any vertex of $W$.

 \item[(5)] $W$ is a twin set of $\overline{G}$.

 \item[(6)] $\tau (G)=\tau(\overline{G})$.

 \item[(7)] $\tau(G)\le \beta_p(G)$.

 \item[(8)] $\tau (G)=\beta_p(G)=n$ if and only if $G$ is the complete graph $K_n$.

 \item[(9)] $\tau (G)=n-1$ if and only if $G$ is the star $K_{1,n-1}$.

 \end{enumerate}
\end{prop}

It is a routine exercise to check all the results showed in Table \ref{tab.cartProductPaths}
 (see also \cite{ChaSaZh00} and the references given in \cite{chmppsw07}).

\begin{table}[h]
 \begin{center}
 \begin{tabular}{|c|cccccc|}
  \hline
   $G$&$P_n$ &$C_{n}$ & $K_{1,n-1}$ & $K_{k,k}$ & $K_{k,n-k}$  & $K_n$  \\
	{\rm order} $n$ & $n\ge 4$ & $n\ge  5$  & $n\ge 3$ & $4\le n=2k$ & $2\le k < n-k$  & $n\ge 2$   \\	
  \hline
 $\beta(G)$  & 1  &  2  &$n-2$  & $n-2$ & $n-2$ &  $n-1$\\
 $\tau(G)$& 1  &  1  &$n-1$  & $k$ & $n-k$ & $n$ \\
 $\beta_p(G)$& 2  &  3  &$n-1$  & $k+1$ & $n-k$ & $n$ \\
  \hline
 \end{tabular}
 \end{center}
 \caption{Metric dimension $\beta$,  twin number $\tau$ and partition dimension $\beta_p$ of  paths, cycles, stars, bicliques and cliques.}
 \label{tab.cartProductPaths}
\end{table}

\begin{figure}[!hbt]
\begin{center}
\includegraphics[width=0.95\textwidth]{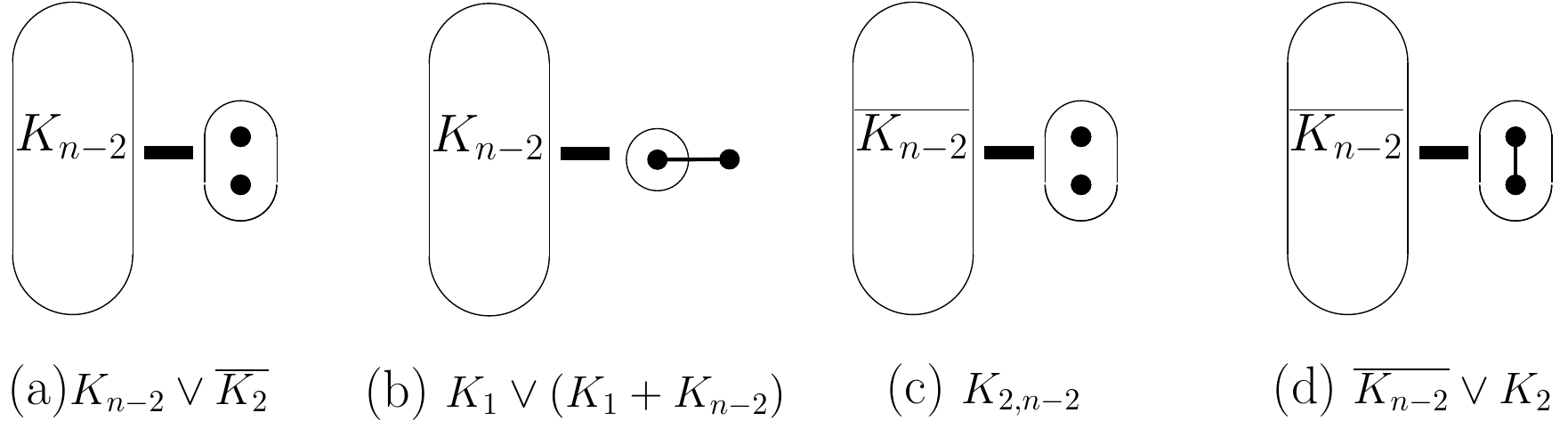}
\caption{Graphs of order $n\ge4$ such that $\tau(G)=n-2$.}\label{pdn-1}
\label{bpn-2}
\end{center}
\end{figure}

We conclude this section by characterizing the set of graphs $G$ such that $\tau (G)=n-2$.

\begin{prop}\label{prop.twingran}
 Let $G=(V,E)$ be a graph of order $n\ge 4$. 
 Then, $\tau (G)=n-2$ if and only if $G$ is one of the following graphs (see Figure \ref{bpn-2}):

   \begin{enumerate}

   \item[{\rm(a)}] the complete split graph $K_{n-2} \vee \overline{K_2}$, obtained by removing an edge from the complete graph $K_n$;
	
   \item[{\rm(b)}] the graph $K_1\vee (K_1+K_{n-2})$, obtained by attaching a leaf to the complete graph $K_{n-1}$;
	
   \item[{\rm(c)}] the complete bipartite graph $K_{2,n-2}$;
	
   \item[{\rm(d)}] the complete split graph $\overline{K_{n-2}} \vee K_2 $.
	
   \end{enumerate}

\end{prop}

\begin{proof} 
It is straightforward to check that the twin number of the four  graphs displayed in Figure \ref{bpn-2} is $n-2$. 
Conversely, suppose that  $G$ is a graph such that $\tau (G)=n-2$.
Let $x,y\in V$  such that $W=V\setminus \{x,y\}$ is the $\tau$-set of $G$. 
Since $G$ is connected, we may suppose  without loss of generality that $W\subseteq N(x)$.
We distinguish two cases.

\noindent \textbf{Case 1}: $G[W] \cong K_{n-2}$.
If $xy\notin E$, then $N(y)=W$, and thus $G\cong \overline{K_2}\vee K_{n-2}$.
If $xy\in E$, then $N(y)=\{x\}$, as otherwise $G\cong K_{n}$, a contradiction.
Thus, $G\cong K_1\vee (K_{n-2}+K_1)$.

\noindent \textbf{Case 2}: $G[W] \cong \overline{K_{n-2}}$.
If $xy\notin E$, then $N(y)=W$, and thus $G\cong K_{2,n-2}$.
If $xy\in E$, then $N(y)=W$, as otherwise $G\cong K_{1,n-1}$, a contradiction. 
Hence, $G\cong \overline{K_{n-2}} \vee K_2 $.
\end{proof}

\section{Twin number versus partition dimension}\label{sec.tnmd}

This section, consisting of 3 subsections, is devoted to obtain relations between the partition dimension $\beta_p(G)$ and the twin number $\tau(G)$  of a graph $G$. 
In the first subsection, a  realization theorem involving both parameters is presented, without any further restriction than the inequality $\tau(G) \le \beta_p(G)$.
The second subsection is devoted to study the parameter $\beta_p(G)$, when $G$ is a graph of order $n$ with "few" twin vertices, to be more precise, such that $\tau(G)\le \frac{n}{2}$.
Finally, the last subsection examines  $\beta_p(G)$, whenever $G$ is a graph for which $\tau(G)> \frac{n}{2}$.

\subsection{Realization Theorem for trees}

A complete $k$-ary tree of height $h$ is a rooted tree whose internal vertices have $k$ children and whose leaves are at distance $h$
from the root. 
Let $T(k,2)$ denote the complete $k$-ary tree of height 2. 
Suppose that  $x$ is the root, $x_1,\dots ,x_k$ are the children of $x$, and $x_{i1},\dots ,x_{ik}$ are the children of $x_i$ for any $i\in \{ 1,\dots ,k\}$  (see Figure \ref{fig:kary}(a)).

\begin{prop}\label{prop.kary}
For any integer $k\ge 2$, $\tau (T(k,2))=k$ and
$\beta_p(T(k,2))=k+1$.
\end{prop}
\begin{proof}
Certainly, $\tau (T(k,2))=k$, and thus $\beta_p(T(k,2))\ge k$.
Suppose that $\beta_p(T(k,2))=k$ and $\Pi=\{ S_1,\dots ,S_k \}$ is a locating partition of size $k$. 
In such a case,  for every $i\in \{ 1,\dots ,k\}$ the vertices $x_{i1},\dots ,x_{ik}$ are twins, and thus each one belongs to a distinct part of $\Pi$. 
So, if $x_r,x_s\in S_i$ for some pair $r,s\in \{1,\dots ,k\}$, with $r\not=s$, then $r(x_r|\Pi)=r(x_s|\Pi)=(1,\dots,1,\underset{i)}0,1\dots ,1)$, which is a contradiction.
Hence, the vertices $x_1,\dots ,x_k$ must belong to distinct parts of $\Pi$. 
We may assume that $x_i\in S_i$, for every $i\in \{ 1,\dots ,k\}$. 
Thus, if $x$ belongs to the part $S_i$, then $r(x|\Pi)=r(x_i|\Pi)=(1,\dots,1,\underset{i)}0,1\dots ,1)$, which is a contradiction.
Hence, $\beta_p(T(k,2))\ge k+1$.
Finally, consider the partition $\Pi=\{ S_1,\dots ,S_k, S_{k+1} \}$ such that $S_{k+1}=\{ x \}$ and, for any $i\in \{ 1,\dots ,k\}$,
$S_i=\{ x_i,x_{1i},x_{2i},\dots ,x_{ki}\}$.
Then, for every $u\in V(T(k,2))$ and for every $i,j,h\in \{ 1,\dots , k\}$ such that $j<i<h$:

		\begin{center}
$r(u|\Pi) = 
  \begin{cases} 
	 (2,\ldots,2,\overset{j)}{1}, 2,\ldots,2,\overset{i)}{0},2,\ldots,2,\overset{h)}{2},2,\ldots,2,2) & \text{if } u=x_{ij}\\
   (2,\ldots,2,\overset{}{2}, \hspace{.1cm}2,\ldots,2,\overset{}{0},\hspace{.03cm}2,\ldots,2,\overset{}{2},\hspace{.07cm}2,\ldots,2,2) & \text{if } u=x_{ii} \\
	 (2,\ldots,2,\overset{}{2}, \hspace{.1cm}2,\ldots,2,\overset{}{0},\hspace{.03cm}2,\ldots,2,\overset{}{1},\hspace{.07cm}2,\ldots,2,2) & \text{if } u=x_{ih} \\
   (1,\ldots,1,\overset{}{1}, \hspace{.1cm}1,\ldots,1,\overset{}{0},\hspace{.03cm}1,\ldots,1,\overset{}{1},\hspace{.07cm}1,\ldots,1,1) & \text{if } u=x_i 
  \end{cases}$ 
\end{center}

Therefore, $\Pi$ is a locating partition, implying that $\beta_p(T(k,2))= k+1$.
\end{proof}

\begin{figure}[ht]
\begin{center}
\includegraphics[width=1\textwidth]{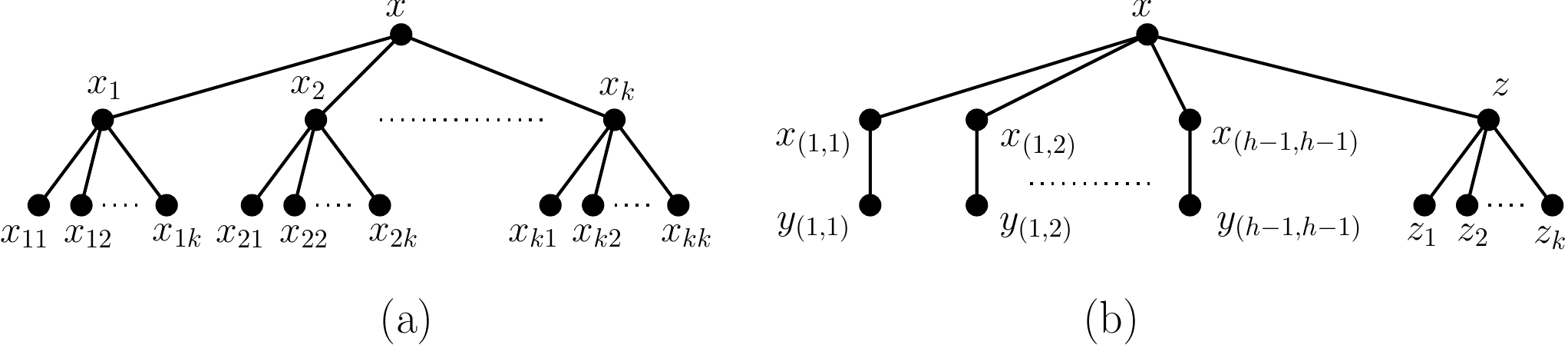}
\caption{The trees (a) $T(k,2)$ and (b) $T^*(k,h)$.}
\label{fig:kary}
\end{center}
\end{figure}

For any $k\ge 1$ and $h\ge 1$, let  $T^*(k,h)$ denote the tree of order $k+2+2(h-1)^2$ defined as follows 
(see Figure \ref{fig:kary}(b)):
$$V(T^*(k,h))= \{ x,z \} \cup \{ z_1,\dots ,z_k\} \cup \{ x_{(i,j)}: 1\le i,j\le h-1 \} \cup \{ y_{(i,j)}: 1\le i,j\le h-1 \},$$
$$E(T^*(k,h))= \{ x x_{(i,j)} : 1\le i,j\le h-1\}\cup \{ x_{(i,j)}y_{(i,j)}: 1\le i,j\le h-1  \} \cup \{ xz, zz_1,\dots ,zz_k\}.$$

\begin{prop}\label{prop.tab}
Let $k,h$ be integers such that $k\ge 1$ and $h\ge k+2$. Then, $\tau (T^*(k,h))=k$ and $\beta_p(T^*(k,h))=h$.
\end{prop}
\begin{proof}
Certainly, $\tau (T^*(k,h))=k$.
Let  $\beta_p(T^*(k,h))=t$.

Next,  we show that $t\ge h$. 
Let $\Pi=\{ S_1,\dots ,S_t\}$ be a locating partition of $T^*(k,h)$.
If there exist two distinct pairs $(i,j)$ and $(i',j')$ such that the vertices  $x_{(i,j)},x_{(i',j')}$ are in the same part of $\Pi$ and $y_{(i,j)},y_{(i',j')}$  are in the same part, then $r(x_{(i,j)}|\Pi)=r(x_{(i',j')}|\Pi)$, which is a contradiction.
Notice that this tree contains $(h-1)^2$ pairs of vertices of the type $(x_{(i,j)},y_{(i,j)})$
and if $t\le h-2$, we achieve at most $(h-2)^2$ such pairs avoiding the preceding condition.
Thus, $t\ge h-1$.
Moreover, if  $t= h-1$, then for every pair $(m,n)\in \{ 1,\dots ,h-1\}^2$, there exists a pair $(i,j)\in \{ 1,\dots ,h-1\}^2$ such that  
$x_{(i,j)}\in S_m$ and 
$y_{(i,j)}\in S_n$.
So, by symmetry, we may assume without loss of generality that $x\in S_1$. Consider the vertices $x_{(i,j)},y_{(i,j)},x_{(i',j')},y_{(i',j')}$ such that $x_{(i,j)}\in S_2$, $y_{(i,j)}\in S_1$ and $x_{(i',j')}\in S_2$, $y_{(i',j')}\in S_2$.
Then $r(x_{(i,j)}|\Pi)=r(x_{(i',j')}|\Pi)=(1,0,2,\dots,2)$, which is a contradiction.
Hence, $t\ge h$.

To prove the equality $t= h$, consider the partition $\Pi=\{S_1,\dots ,S_h\}$ such that:
$$\left.
       \begin{array}{ll}
S_i=\{ x_{(i, m)} : 1\le m \le h-1\} \cup \{  y_{(n,i)} : 1\le n\le h-1 \}\cup \{ z_i\},  \phantom{i}\hbox{ if $1\le i \le k$}\\
S_i=\{ x_{(i, m)} : 1\le m \le h-1\} \cup \{  y_{(n,i)} : 1\le n\le h-1 \},  \phantom{xxxxxi}\hbox{ if $k< i \le h-1$} &\\
S_h=\{ x,z \}. &
\end{array}
\right.
$$
Let $i\in \{ 1,\dots , h-1\}$. Then, for every $m,n\in  \{ 1,\dots , h-1\}$, $m,n\not= i$:
\begin{center}
$r(u|\Pi) =
  \begin{cases}
    (2,\ldots,2,\overset{i)}{0},2,\ldots,2,\overset{m)}{1},2,\ldots,2,1) & \text{if } u=x_{(i,m)}\\
   (2,\ldots,2,{0},2,\ldots,2,{\, 2\,},2,\dots,2,1) & \text{if } u=x_{(i,i)}\\
\end{cases}$
$r(u|\Pi) =
  \begin{cases}
    (3,\ldots,3,\overset{i)}{0},3,\ldots,3,\overset{n)}{1},3,\dots,3,2) & \text{if } u=y_{(n,i)}\\
    (3,\ldots,3,{0},3,\ldots,3,{3},3,\dots,3,2) & \text{if } u=y_{(i,i)}
   \end{cases}$
\end{center}
Therefore, $r(u,|\Pi)\not= r(v|\Pi)$ if  $u,v\in \{ x_{(i, m)} : 1\le m \le h-1\} \cup \{  y_{(n,i)} : 1\le n\le h-1 \}$ and $u\not= v$.
Moreover,  it is straightforward to check that, if $i\in \{ 1,\dots , k\}$, then for every  $u\in S_i$, $u\not= z_i$, we have
   $$r(z_i|\Pi)=(2,\ldots,2,\overset{i)}{0},2,\ldots,\overset{k)}{2},3,\dots,3,1)\not= r(u|\Pi)\, .$$
Finally,  for $x,z\in S_h$, we have
$$r(x|\Pi)=(1,\ldots,\overset{k)}{1},1,\ldots,1,0)\not= (1,\ldots,\overset{k)}{1},2,\dots,2,0)=r(z|\Pi).$$
Therefore, $\Pi$ is a locating partition, 
implying that $\beta_p(T^*(k,h))=h$.
\end{proof}

\begin{thm}
Let $a,b$ be integers such that $1\le a\le b$. Then, there exists a tree $T$ such that $\tau (T)=a$ and $\beta_p(T)=b$.
\end{thm}

\begin{proof}
For $a=b=1$, the trivial graph $P_1$ satisfies $\tau (P_1)=\beta_p(P_1)=1$. 
For $a=b\ge 2$, consider the star $K_{1,a}$.
For $a=1$ and $b=2$, take the path $P_4$.
If $2\le a$ and $b=a+1$, consider the tree $T(a,2)$ studied in Proposition \ref{prop.kary}.
Finally, if $a \ge 1$ and $b\ge a+2$, take the tree $T^*(a,b)$ analyzed in Proposition \ref{prop.tab}.
\end{proof}

\subsection{Twin number at most half the order}

In this subsection, we approach the case when $G$ is a graph of order $n$ such that $\tau (G)=\tau \le \frac{n}{2}$. 
Concretely, we prove that, in such a case, $\beta_p(G)\le n-3$.

\begin{lem}\label{lem.completempty}
Let $D$ be a subset of vertices of size $k\ge 3$  of a graph $G$ such that $G[D]$ is neither complete nor empty. 
Then, there exist at least three different vertices $u,v,w\in D$ such that $uv\in E(G)$ and $uw\notin E(G)$
\end{lem}
\begin{proof}
If $G[D]$ is neither complete nor empty, then there is at least one vertex $u$ such that  $1\le \deg_{G[D]}(u) \le k-2$.
Let  $v$ (resp.  $w$) be a a vertex adjacent (resp. non-adjacent) to $u$. Then, $u,v,w$ satisfy the desired condition.
\end{proof}

\begin{lem}\label{lem.grau}
If $G$ is a  nontrivial graph of order $n$ with  a vertex $u$ of degree $k$, then $\beta_p(G)\le n-\min \{k,n-1-k\}$.
\end{lem}
\begin{proof}
Let $N(v)=\{x_1,\dots ,x_k\}$ and $V(G)\setminus N(v)=\{y_1, \dots ,y_{n-1-k}\}$ and $m=\min \{ k, n-1-k\}$. 
Take the partition $\Pi= \{ S_1,\dots ,S_m \} \cup \{ \{ z\} : z\notin  S_1\cup \ldots \cup S_m  \}\}$, where
$S_i=\{x_i,y_i\}$ for $i=1,\dots ,m $. 
Observe that $\{v\}$ resolves the vertices of $S_i=\{x_i,y_i\}$ for $i=1,\dots ,m $.
Therefore, $\Pi$ is a locating partition of $G$, impliying that  $\beta_p(G)\le |\Pi|=n-m=n-\min \{ k, n-1-k\}$.
\end{proof}

\begin{cor}\label{cor.grau}
If $G$ is a  graph of order $n\ge 7$ with at least one vertex $u$ satisfying $3\le \deg(u)\le n-4$, then $\beta_p(G)\le n-3$.
\end{cor}

As a direct consequence of Theorem \ref{mdpd}, we know that if $G$ is a graph such that ${\rm diam}(G)\ge4$, then $\beta_p(G)\le n-3$. 
Next, we study the cases ${\rm diam}(G)=3$ and ${\rm diam}(G)=2$.

\begin{prop}\label{prop.twinpetitdiam3}
Let $G$ be a  graph of order $n\ge 9$. If  $\tau(G)\le \frac{n}{2}$ and ${\rm diam}(G)=3$, then  $\beta_p(G)\le n-3$.
\end{prop}
\begin{proof}
By Corollary \ref{cor.grau}, and having also in mind that $G$ has no universal vertex, since its diameter is 3, we may suppose that, for every vertex $w$, $\deg (w)\in \{ 1,2,n-3,n-2\}$.
Let $u$ be a vertex of eccentricity $3$.
Consider the nonempty subsets $D_i=\{ v : d(u,v)=i \}$ for $i=1,2,3$. 
If at most one of these three subsets has exactly one vertex, then there exist five distinct vertices $x_1,x_2,x_3,y_{1},y_{2}$ such that,  for $i=1,2,3$, $x_i\in D_i$ and vertices $y_1$ and $y_2$ do not belong to the same set $D_i$.   
Consider the partition 
$\Pi= \{  S_1, S_2  \} \cup \{ \{ z\} : z\notin  S_1\cup S_2   \}\}$, where
$S_1=\{ x_1,x_2,x_3\}$ and $S_2=\{y_{1},y_{2}\}$. 
Then, $\{ u \}$ resolves every pair of vertices in $S_1$ and the vertices in $S_2$.
Therefore, $\Pi$ is a locating partition, implying that $\beta_p(G)\le n-3$.

Next,  suppose that $|D_{i_0}|=n-3$ for exactly one value $i_0\in \{ 1,2,3\}$ and $|D_i|=1$ for $i\not= i_0$. 
We distinguish two cases.

\begin{enumerate}[(1)]

\item $G[D_{i_0}]$ is neither complete nor empty.
Then by Lemma \ref{lem.completempty}, there exist vertices $r,s,t\in D_{i_0}$ such that $rs\in E(G)$ and $rt\notin E(G)$. 
Consider the sets $S_1=\{ s,t \}$ and $S_2=\{x_1,x_2,x_3\}$, where  $x_i\in D_i$ for $i=1,2,3$, with the additional condition $S_2\cap \{ r,s,t \}=\emptyset$, which is possible since $|D_{i_0}|\ge 4$.
Take the partition
$\Pi= \{  S_1, S_2  \} \cup \{ \{ z\} : z\notin  S_1\cup S_2   \}\}$.
Observe that $\{ r \}$ resolves the vertices in $S_1$ and $\{ u \}$ resolves every pair of vertices in $S_2$.
Therefore, $\Pi$  is a locating partition, implying that $\beta_p(G)\le n-3$.

\item $G[D_{i_0}]$ is either complete or empty. 
We distinguish three cases, depending on for which $i_0\in \{ 1,2,3\}$, $|D_{i_0}|=n-3$.

\begin{enumerate}[(a)]

\item $|D_3|=n-3$.  Then, $D_3$ is a twin set with $n-3$ vertices, a contradiction as $n\ge9$.

\item $|D_1|=n-3$.
Let $v$ be the (unique) vertex of $D_2$.
Then $D_1\cap N(v)$ and $D_1\cap \overline{N(v)}$ are twin sets.
If $\deg(v)=2$, then $|D_1\cap \overline{N(v)}|=n-4$, a contradiction.
If $n-3 \le \deg(v) \le n-2$, then $|D_1\cap N(v)|\ge n-4$, again a contradiction.

\item $|D_2|=n-3$.
Let $v$ be the (unique) vertex of $D_3$. 
Then, both $N(v)$ and $D_2 \setminus N(v)$ are twin sets.
Notice that $\deg (v)\in \{ 1,2,n-3 \}$. We distinguish cases.

\begin{enumerate}[(c.i)]

\item If $\deg (v)=1$ (resp. $\deg (v)=n-3$) , then $|D_2 \setminus N(v)|=n-4$ (resp. $|N(v)|=n-3$), a contradiction.

\item If $\deg (v)=2$, then $|D_2 \setminus N(v)|=n-5$. 
Let $N(v)=\{a_1,a_2\}$, $D_2\setminus N(v)=\{b_1,\ldots,b_{n-5}\}$, $D_1=\{x\}$.
Take the partition $\Pi= \{  S_1, S_2, S_3  \} \cup \{ \{ z\} : z\notin  S_1\cup S_2\cup S_3   \}\}$, 
where  $S_1=\{a_1,b_1\}$, $S_2=\{a_2,b_2\}$ and $S_3=\{x,b_3\}$.
Observe that $\{ v \}$ resolves the vertices in $S_1$ and $S_2$ and  $\{ u \}$ resolves the vertices in $S_3$.
Therefore, $\Pi$  is a locating partition, implying that $\beta_p(G)\le n-3$.

\end{enumerate}

\end{enumerate}

\end{enumerate}

\vspace{-1.0cm}\end{proof}

\begin{prop}\label{prop.twinpetitdiam2}
Let $G$ be a  graph of order $n\ge 9$. 
If  $\tau(G)\le \frac{n}{2}$ and ${\rm diam}(G)=2$, then  $\beta_p(G)\le n-3$.
\end{prop}

\begin{proof}
By Corollary \ref{cor.grau}, we may suppose that, for every vertex $w\in V(G)$, $\deg (w)\in \{ 1,2,n-3,n-2,n-1\}$.
We distinguish three cases.

\begin{enumerate}[(i)]

\item \emph{There exists a vertex $u$ of degree $2$.}
Consider the subsets $D_1=N(u)=\{ x_1,x_2\}$ and $D_2=\{ v : d(u,v)=2\}$. 
We distinguish two cases.

\begin{enumerate}[(1)]

\item $G[D_{2}]$ is neither complete nor empty.
Then, Lemma \ref{lem.completempty}, there exist three different vertices $r,s,t\in D_2$ such that $rs\in E(G)$ and $rt\notin E(G)$. 
Consider two different vertices 
$y_1,y_2\in D_2\setminus \{ r,s,t \}$ and let $S_1=\{ x_1,y_1\}$, $S_2=\{ x_2,y_2 \}$, $S_3=\{ s,t \}$. 
Then, $\{ u \}$ resolves the vertices in $S_1$ and in $S_2$, and $\{r \}$ resolves the vertices in $S_3$.
Hence, $\Pi = \{ S_1,S_2,S_3 \}\cup \{ \{ z \} : z\notin S_1\cup S_2 \cup S_3 \}$ is a locating partition of $G$.

\item $G[D_{2}]$ is either complete or empty.
 Then the subsets of $D_2$, 
$A=N(x_1)\cap\overline{N(x_2)}\cap D_2$, $B=N(x_1)\cap N(x_2)\cap D_2$ and $C=\overline{N(x_1)}\cap N(x_2)\cap D_2$  are twin sets. We distinguish cases.

\begin{enumerate}[(a)]

 \item  If either $\deg(x_1)\le 2$ or $\deg(x_2)\le 2$, then either $|C|\ge n-4$ or $|A|\ge n-4$, in both cases a contradiction as $n\ge9$.
  
 \item If both $x_1$ and $x_2$ have degree at least $n-3$, then $0 \le |A| \le 2$,  $0 \le |C| \le 2$ and $n-7 \le |B| \le n-3$.
We distinguish cases, depending on the size of $B$.

\begin{enumerate}[(b.1)]

\item $|B| \ge n-4$. Then, $\tau(G) \ge n-4$, a contradiction as $n\ge9$.

\item $|B| = n-5 \ge 4$. 
If $D_2 \cong \overline {K_{n-3}}$, then $\tau(G)=n-4$, as $B\cup \{u\}$ is a (maximum) twin set of $G$.
Suppose that  $D_2 \cong K_{n-3}$.
Let $A\cup C=\{y_1,y_2\}$, $\{b_1,b_2,b_3,b_4\}\subseteq B$.
Consider the partition $\Pi= \{  S_1, S_2, S_3  \} \cup \{ \{ z\} : z\notin  S_1\cup S_2\cup S_3   \}\}$, 
where  $S_1=\{b_1,y_1\}$, $S_2=\{b_2,y_2\}$ and $S_3=\{b_3,u\}$.
Observe that either $\{ x_1\}$ or  $\{ x_2\}$ resolves the vertices of $S_1$ and $S_2$.
Notice also that $\{ b_4\}$ resolves the vertices in $S_3$.
Hence, $\Pi$ is a locating partition of $G$.

\item $2 \le n-7 \le |B| \le n-6$. 
We may assume without loss of generality that $|A|=2$ and $1 \le |C| \le 2$.
Let $A=\{a_1,a_2\}$, $\{b_1,b_2\}\subseteq B$, $\{c_1\}\subseteq C$.
Consider the partition $\Pi= \{  S_1, S_2, S_3  \} \cup \{ \{ z\} : z\notin  S_1\cup S_2\cup S_3   \}\}$, 
where  $S_1=\{a_1,b_1\}$, $S_2=\{a_2,b_2\}$ and $S_3=\{x,c_1\}$.
Observe that $\{ x_2\}$ resolves the vertices of $S_1$ and $S_2$.
Notice also that $\{ u\}$ resolves the vertices in $S_3$.
Hence, $\Pi$ is a locating partition of $G$.

\end{enumerate}

\end{enumerate}

\end{enumerate}

\item \emph{There exists at least one vertex $u$ of degree $1$ and there is no vertex of degree 2.}
In this case, the neighbor $u$ of $v$ is a universal vertex $v$.
Let $\Omega$ be the set of vertices different from $v$ that are not leaves.
Notice that there are at most two vertices of degree 1 in $G$, as otherwise all vertices in $\Omega$ would have degree between  $3$ and  $n-4$, contradicting the assumption made at the beginning of the proof.

If there are exactly two vertices of degree 1, then $|\Omega|=n-3$.
In such a case,  $\Omega$ induces a complete graph in $G$, 
as otherwise the non-universal vertices in $G[\Omega]$ would have degree at most $n-4$.
So,  $\Omega$ is a twin set, implying that  $\tau (G) =n-3 > \frac{n}{2} $, a contradiction.

Suppose thus that $u$ is the only vertex of degree 1, which means that  $\Omega$ contains $n-2$ vertices,
 all of them  of degree $n-3$ or $n-2$.
Consider the graph   $H=\overline{G}[\Omega]$.
Certainly,  $H$ has $n-2$ vertices, all of them  of degree $0$ or $1$.
Let $H_i$ denote the set of vertices of degree $i$ of $H$, for $i=0,1$.
Observe that $|H_0|\le \frac{n}{2}$, since $H_0$ is a twin set in $G$.
Hence, $|H_1|\ge4$, as  $n\ge9$ and the size of $H_1$ must be even.
We distinguish two cases, depending on the size of $H_1$.

\begin{enumerate}[(a)]

\item $|H_1|=4$. 
Notice that $|H_0|\ge3$.
Let $\{y_1,y_2,y_3\}\subseteq H_0$ and $H_1=\{x_1,x_2,x_3,x_4\}$ such that $\{x_1x_2,x_3x_4\}\subseteq E(H_1)$.
Consider the partition $ \Pi =\{ S_1,S_2,S_3\} \cup \{ \{ z \} : z\notin S_1\cup S_2 \cup S_3 \}$, 
where  $S_1= \{x_1,y_1\}$, $S_2=\{ x_2,y_2\}$ and $S_3=\{ u,v\}$.
Observe that 
$d_G(x_2,x_1)=2\not= 1=d_G(x_2,y_1)$, 
$d_G(x_4,x_3)=2\not= 1=d_G(x_4,y_3)$,
$d_G(y_2,u)=2\not= 1=d_G(y_2,v)$.
Hence, 
$\{ x_2 \}$ resolves the vertices in $S_1$, $\{ x_4 \}$ resolves the vertices in $S_2$  and $\{ y_2 \}$ resolves the vertices in $S_3$.
Therefore, $\Pi$ is a locating partition of $G$, implying that $\beta_p(G)\le n-3$.

\item $|H_1|\ge6$. 
Let $\{x_1,x_2,x_3,x_4,x_5,x_6\}\subseteq H_1$ such that $\{x_1x_2,x_3x_4,x_5x_6\}\subseteq E(H_1)$.
Consider the partition $ \Pi =\{ S_1,S_2,S_3\} \cup \{ \{ z \} : z\notin S_1\cup S_2 \cup S_3 \}$, 
where  $S_1= \{u,v\}$, $S_2=\{ x_2,x_4\}$ and $S_3=\{ x_3,x_5\}$.
Observe that 
$d_G(u,x_1)=2\not= 1=d_G(v,x_1)$, 
$d_G(x_4,x_1)=2\not= 1=d_G(x_2,x_1)$,
$d_G(x_3,x_6)=2\not= 1=d_G(x_5,x_6)$.
Hence, 
$\{ x_1 \}$ resolves the vertices in $S_1$ and in $S_2$, and $\{ x_6 \}$ resolves the vertices in $S_3$.
Therefore, $\Pi$ is a locating partition of $G$, implying that $\beta_p(G)\le n-3$.

\end{enumerate}

\item \emph{There are no vertices of degree at most 2.} 
In this case, all the vertices of $G$ have degree $n-3$, $n-2$ or $n-1$, that is to say, 
all the vertices of  $\overline{G}$  have degree 0, 1 or 2.
Since $G$ has at most $\frac{n}{2}$ pairwise twin vertices, there are at most $\frac{n}{2}$ vertices of degree 0 in $\overline{G}$.
Let $H_i$ denote the set of vertices of degree  $i$ of $\overline{G}$, for $i=0,1,2$.
Let $\Gamma=H_1\cup H_2$ and $H=\overline{G}[\Gamma]$.
We distinguish  cases, depending on the size of $\Gamma$, showing in each of them a collection of three  2-subsets 
$S_1,S_2,S_3$ such that the corresponding partition
 $\Pi =\{ S_1,S_2,S_3 \}\cup \{ \{ z \} : z\notin S_1 \cup S_2\cup S_3 \}$
is a locating partition for $G$, implying thus that $\beta_p(G)\le n-3$.

\begin{enumerate}[(a)]

\item  $|\Gamma|\in \{ 5, 6 \}$. Then, $|H_0|\ge 3$. Let $\{ y_1,y_2,y_3 \}\subseteq H_0$.
It easy to check that in both cases $H$ contains at least three edges either of the form
(i) $x_1x_2$, $x_3x_4$, $x_4x_5$  or of the form (ii) $x_1x_2$, $x_3x_4$, $x_5x_6$.
Take $S_1=\{ x_1,y_1\}$, $S_2= \{ x_3,y_2\}$, $S_3=\{ x_5,y_3\}$.
Notice that,  in case (i),  $\{ x_2 \}$ resolves  the vertices in $S_1$ and $\{ x_4\}$ resolves the vertices in $S_2$ and in $S_3$,
and in case (ii),  $\{ x_2 \}$ resolves the vertices in $S_1$, $\{ x_4\}$ resolves the vertices in $S_2$ and   $\{ x_6\}$ resolves the vertices in $S_3$.

\item $|\Gamma|=7$. 
Then,  $|H_0|\ge 2$. 
Let $\Gamma=\{x_1,x_2,x_3,x_4,x_5,x_6,x_7\}$ 
such that $\{x_1x_2, x_2x_3, x_4x_5,x_6x_7\}\subseteq E(\overline{G})$, $x_4x_6\notin E(\overline{G})$  and $\{y_1,y_2\}\subseteq H_0$ .
Take  $S_1= \{x_1,y_1\}$, $S_2=\{ x_3,y_2\}$ and $S_3=\{x_5,x_6\}$.
Observe that $\{ x_2 \}$ resolves the vertices in $S_1$ and in $S_2$, and $\{ x_4 \}$ resolves the vertices in $S_3$.

\item $|\Gamma| \ge 8$. 
Then,   all the connected components of $H$ are isomorphic either to a path or to a cycle.
We distinguish cases, depending on the number of components of $\overline{G}[\Gamma]$.

\begin{enumerate}[(c.1)]

  \item  If  $H$  is connected, then $H$ contains a path $x_1x_2\dots x_8$ of length 7. 
	Take $S_1=\{x_1,x_2\}$,  $S_2=\{x_4,x_5\}$ and  $S_3=\{x_7,x_8\}$. 
	Then, $\{ x_3 \}$ resolves the vertices in $S_1$ and in $S_2$,
  and $\{ x_6 \}$ resolves the vertices in $S_3~$.

  \item  If  $H$  has 2 connected components, say $C_1$ and $C_2$, let assume that $|V(C_1)|\ge |V(C_2)|\ge 2$. 
	We distinguish two cases.
	
   \begin{enumerate}[(c.2.1)]
	
       \item If one of the connected components has at most 3 vertices, then $|V(C_1)|\ge 5$ and $2\le |V(C_2)|\le 3$. 
			If   $x_1x_2x_3x_4x_5$ and $y_1y_2$ are paths contained in  $C_1$ and $C_2$, respectively, then consider $S_1=\{x_1,y_1\}$,  $S_2=\{x_3,y_2\}$ and  $S_3=\{x_5,t\}$, where $t$ is any vertex different from $x_1,x_2,x_3,x_4,x_5,y_1,y_2$.
         Then,  it is easy to check that $\{ x_2 \}$ resolves the vertices in $S_1$ and in $S_2$, and $\{ x_4 \}$ resolves the vertices in $S_3$.
				
     \item If both connected components have at least 4 vertices, let  $x_1x_2x_3x_4$ and $y_1y_2y_3y_4$ be paths contained in $C_1$ and $C_2$, respectively. 
		Take $S_1=\{x_1,y_1\}$,  $S_2=\{x_3,y_2\}$ and  $S_3=\{x_4,y_4\}$.
    Then, it is easy to check that $\{ x_2 \}$ resolves the vertices in $S_1$ and in $S_2$, and $\{ y_3 \}$ resolves the vertices in $S_3$.
			
   \end{enumerate}
	
  \item  If $H$ has 3 connected components, say $C_1$, $C_2$ and $C_3$, 
	then we may assume that $|V(C_1)|\ge 3$ and $|V(C_1)|\ge |V(C_2)|\ge |V(C_3)|\ge 2$.
  Let $x_1x_2x_3$, $y_1y_2$ and $z_1z_2$ be paths contained in $C_1$, $C_2$  and $C_3$ respectively. 
	Take $S_1=\{x_1,y_1\}$,  $S_2=\{x_3,y_2\}$ and  $S_3=\{z_2,t\}$, where $t$ is a vertex in $C_1\cup C_2$ different from $x_1,x_2,x_3,y_1,y_2$ that exists because $|\Gamma|\ge 8$.
  Then, it is easy to check that $\{ x_2 \}$ resolves the vertices in $S_1$ and in $S_2$, and $\{ z_1 \}$ resolves the vertices in $S_3$.
			
  \item  If $H$ has at least 4 connected components, say $C_1$, $C_2$, $C_3$ and $C_4$, 
	then we may assume that  $|V(C_1)|\ge |V(C_2)|\ge |V(C_3)|\ge  |V(C_4)|\ge 2$. 
  Let $x_1x_2$, $y_1y_2$, $z_1z_2$ and $t_1t_2$ be edges of $C_1$, $C_2$, $C_3$  and $C_4$ respectively. 
	Take $S_1=\{x_1,y_1\}$,  $S_2=\{y_2,z_2\}$ and  $S_3=\{t_1,w\}$, where $w$ is a vertex in $C_1\cup (V\setminus \Gamma)$ different from $x_1,x_2$, that  exists since $G$ has order at least $9$.
  Then, it is easy to check that $\{ x_2 \}$ resolves the vertices in $S_1$, $\{ z_1 \}$ resolves the vertices in $S_2$, and $\{ t_2 \}$ resolves the vertices in $S_3$.
			
\end{enumerate}
\end{enumerate}
\end{enumerate}

\vspace{-1.2cm}\end{proof}

As a consequence of Theorem \ref{mdpd}, Proposition \ref{prop.twinpetitdiam3} and Proposition \ref{prop.twinpetitdiam2}, the following result is obtained.

\begin{thm}\label{thm.twinpetit}
Let $G$ be a  graph of order $n\ge 9$. If  $\tau(G)\le \frac{n}{2}$, then  $\beta_p(G)\le n-3$.
\end{thm}

\subsection{Twin number greater than half the order}

In this subsection, we focus our attention  on the case when $G$ is a nontrivial graph of order $n$ such that $\tau (G)=\tau > \frac{n}{2}$. 
Notice that, in these graphs there is a unique $\tau$-set $W$.
Among others, we prove that, in such a case, $\displaystyle \tau  (G)\le  \beta_p(G)  \le \frac{n+\tau(G)}{2}$.

\begin{prop}\label{cal}
Let $G$ be a graph of order $n$, other than $K_n$. 
If $W$ is a $\tau$-set such that  $G[W]\cong  {K_{\tau}}$, then $\beta_p(G)\ge \tau(G) +1$.
\end{prop}

\begin{proof}
Assume that $\Pi=\{S_1,S_2,\ldots,S_{\tau}\}$ is  a locating partition of $G$.
If $W=\{w_1,w_2,\ldots,w_{\tau}\}$,  then  we can assume without loss of generality that, for every $i\in\{1,...,\tau\}$, $w_i\in S_i$.
Let $v$ be a vertex of $N(W)\setminus W$. 
Take $j\in\{1,...,\tau\}$ such that $\{w_j,v\}\subseteq S_j$.
Certainly, $r(v|\Pi)=(1,\dots 1,\underset{j)}{0}, 1, \ldots, 1)=r(w_j|\Pi)$, a contradiction.
\end{proof}

\begin{thm}\label{kmedios}
Let $G=(V,E)$ be a  graph of order $n$ such that $\frac{n}{2} < \tau (G)=\tau= n-k$ and let $W$ be its $\tau$-set. 
If $G[W]\cong {K_{\tau}}$, then $\beta_p(G)\le n-k/2$.
\end{thm}

\begin{proof}
Let $W=\{ w_1,\dots , w_{\tau}\}$, $W_1=N(W)\setminus W=\{ x\in V(G) : d(x,W)=1 \}$ and $W_2=V\setminus N[W]=\{ x\in V(G) : d(x,W)\ge 2 \}$, 
and denote $r=|W_1|$, $t=|W_2|$. Observe that $\{ W , W_1 , W_2 \} $ is a partition of $V(G)$ and $k=r+t$.
Since $W$ is a set of twin vertices, we have that $xy\in E(G)$ for all $x\in W$ and $y\in W_1$.

\begin{figure}[htb]
\begin{center}
\includegraphics[width=0.38\textwidth]{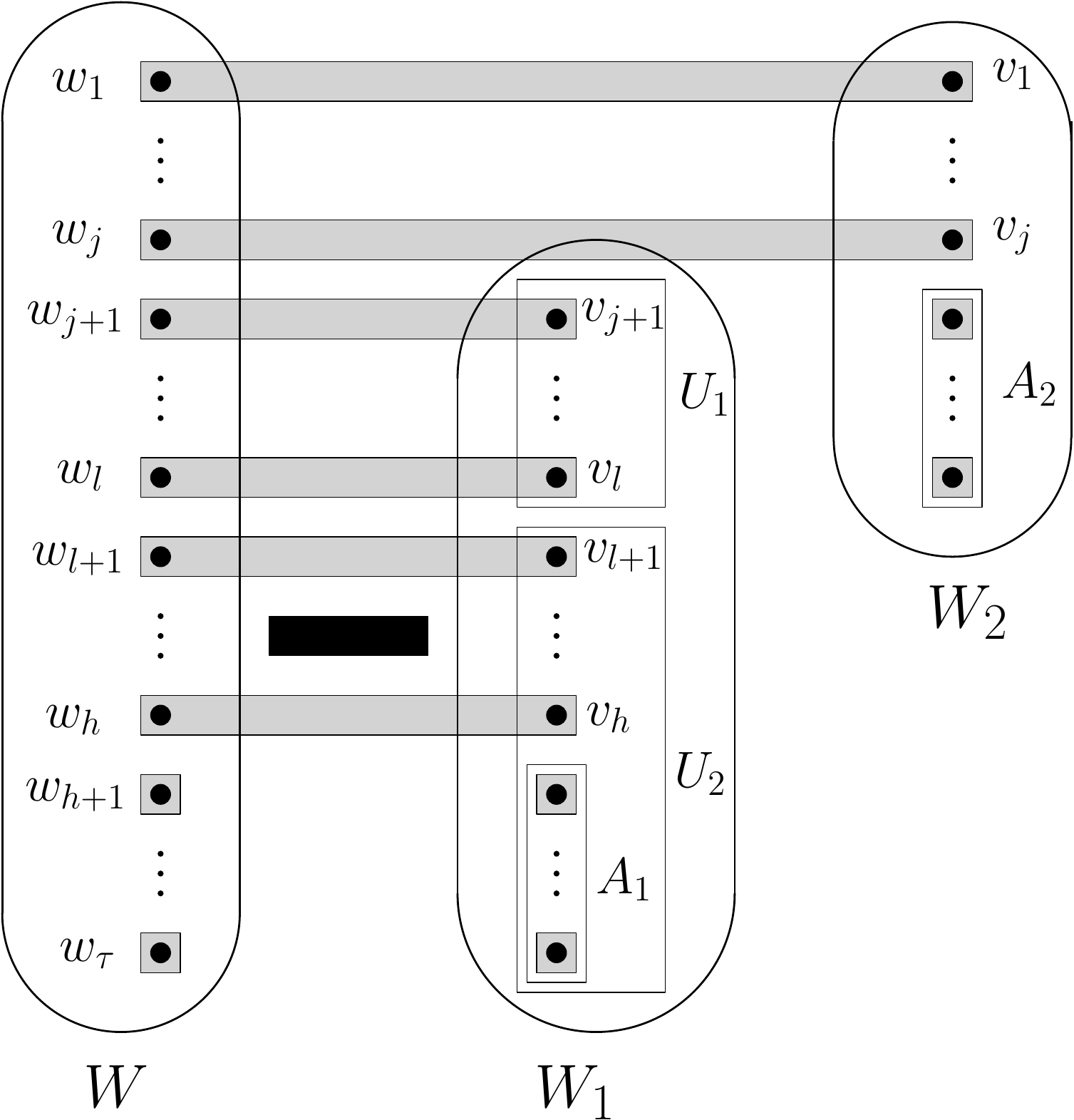}
\caption{In this figure, $W_1=N(W)\setminus W$ and $W_2=V\setminus N[W]$.}
\label{partsofpi}
\end{center}
\end{figure}

Consider the subsets $U_1=\{  x\in W_1 : \deg_{G[W_1]}(x)=r-1\}$ and $U_2=W_1\setminus U_1$ of $W_1$.
If $x\in U_1$, then there exists at least one vertex $y\in W_2$ such that $x y\in E(G)$, otherwise $x$ should be in $W$. 
Let us assign to each vertex $x\in U_1$  one vertex $y(x)\in W_2$ such that $x \, y(x)\in E(G)$ and consider the set $A_2=\{ y(x) : x\in U_1\}\subseteq W_2$. 
Observe that for different vertices $x,x'\in U_1$, the vertices $y(x)$, $y(x')$ are not necessarily different. By construction, $|A_2|\le |U_1|$.

Next, consider the subgraph $G[U_2]$ induced by the vertices of $U_2$. If $s=|U_2|$, then by definition, this subgraph has maximum degree at most $s-2$, and hence, the complement $\overline{G[U_2]}$ has minimum degree at least 1. 
It is well known that every graph without isolated vertices contains a dominating set of cardinality at most half the order
 (see \cite{ore}). 
Let $A_1$ be a dominating set of $\overline{G[U_2]}$ with $|A_1|\le s/2$.

If ($W_1\cup W_2) \setminus (A_1\cup A_2)= \{ v_1, \dots , v_h\}$, we show that the partition $\Pi$ defined as follows is a locating partition of $G$ (see Figure \ref{partsofpi}):
$$\Pi =
\{ \{ x \} : x\in A_1\} \cup
\{ \{ y \} : y\in A_2\} \cup
\{ \{ w_i,v_i \} : 1\le i\le h\} \cup
\{ \{ w_i \} : h+1 \le i \le \tau\} .
$$
Observe that $\Pi$ is well defined since $h< k< \frac{n}{2} \le \tau$.
To prove this claim, it is sufficient to show that for every $i\in \{ 1,\dots ,h\}$ there exists a part of $\Pi$ at different distance from $w_i$ and $v_i$. 
We distinguish the following cases:

\begin{enumerate}[i)]

 \item If $v_i\in U_1$, consider the vertex $y(v_i)\in A_2$  such that $v_i \, y(v_i)\in E(G)$. 

Then, $ d(w_i,  \{ y(v_i) \})= 2\not= 1=d(v_i, \{ y(v_i) \})$.
	
	\item If $v_i\in U_2\setminus A_1$, consider a vertex  $x\in A_1$ dominating $v_i$ in $\overline{G[U_2]}$, i.e., $x\, v_i\notin E(G)$.  
	
	Then, $ d(w_i,  \{ x \})= 1<d(v_i, \{ x \})$.

  \item If $v_i\in W_2\setminus A_2$, then $ d(w_i,  \{ w_{\tau} \})= 1<2\le d(v_i, \{ w_{\tau} \})$.
	
\end{enumerate}

Observe that 
$|A_2|\le |U_1| =r-s$ and $|A_2|\le |W_2| = t$, so we can deduce that
$|A_2|\le (r-s+t)/2$. Therefore, the partition dimension of $G$ satisfies:

  $\beta_p(G) \le |\Pi|=n- |(W_1\cup W_2)\setminus (A_1\cup A_2)|=n-[(r+t)- (|A_1|+ |A_2|)]=  n-k/2$
\end{proof}

\begin{prop}\label{cal0}
Let $G=(V,E)$ be a graph of order $n$ such that $\tau (G)=\tau > \frac{n}{2}$. 
If  its $\tau$-set $W$ satisfies $G[W]\cong \overline{K_{\tau}}$, then $\beta_p(G)=\tau$.
\end{prop}

\begin{proof}
Let $W=\{ w_1,\dots ,w_{\tau}\}$, $V\setminus W=\{ v_1,\dots ,v_s\}$ and $N(W)=\{  v_1,\dots ,v_r \}$, where
$1\le r\le s< \tau$.
By Proposition \ref{twin.list}(7), $\beta_p(G) \ge \tau$. 
To prove the equality, consider the partition $\Pi =\{S_1,\dots ,S_{\tau} \}$, where
$Si=\{w_i,v_i\}$ if $1\le i\le s$, and $S_i=\{w_i\}$ if $s< i\le \tau$.
Observe that for any $i,j\in \{1,\dots ,\tau\}$ with $i\not= j$, $h\in \{1,\dots ,r\}$ and $k\in \{r+1,\dots ,s\}$,
$d(w_i,w_j)=2$, $d(v_h,w_j)=1$ and $d(v_k,w_j)=2$.
To prove that $\Pi$ is a locating partition of $G$,  we distinguish  cases.

 \textbf{Case 1}: $1\le i\le r$. Then, $d(w_i,S_{\tau})=d(w_i,w_{\tau})=2 \neq 1= d(v_i,w_{\tau})=d(v_i,S_{\tau})$.

 \textbf{Case 2}: $r<i\le s$. We distinguish two cases.

  \textbf{Case 2.1}: For some $k \in \{1,\ldots, r\}$, $v_iv_k\notin E(G)$.
	Consider the part $S_k=\{ w_k,v_k\}$. 
	On the one hand,  $d(w_i,S_k)=1$ since $d(w_i,v_k)=1$. 
	On the other hand,  $d(v_i,S_k)\ge 2$ since $d(v_i,w_k)\ge 2$ and $d(v_i,v_k)\ge 2$. 
	Therefore, $d(w_i,S_{k})\not= d(v_i,S_{k})$  (see Figure \ref{fig.twins}(a)).
	
\begin{figure}[htb]
\begin{center}
\includegraphics[width=0.7\textwidth]{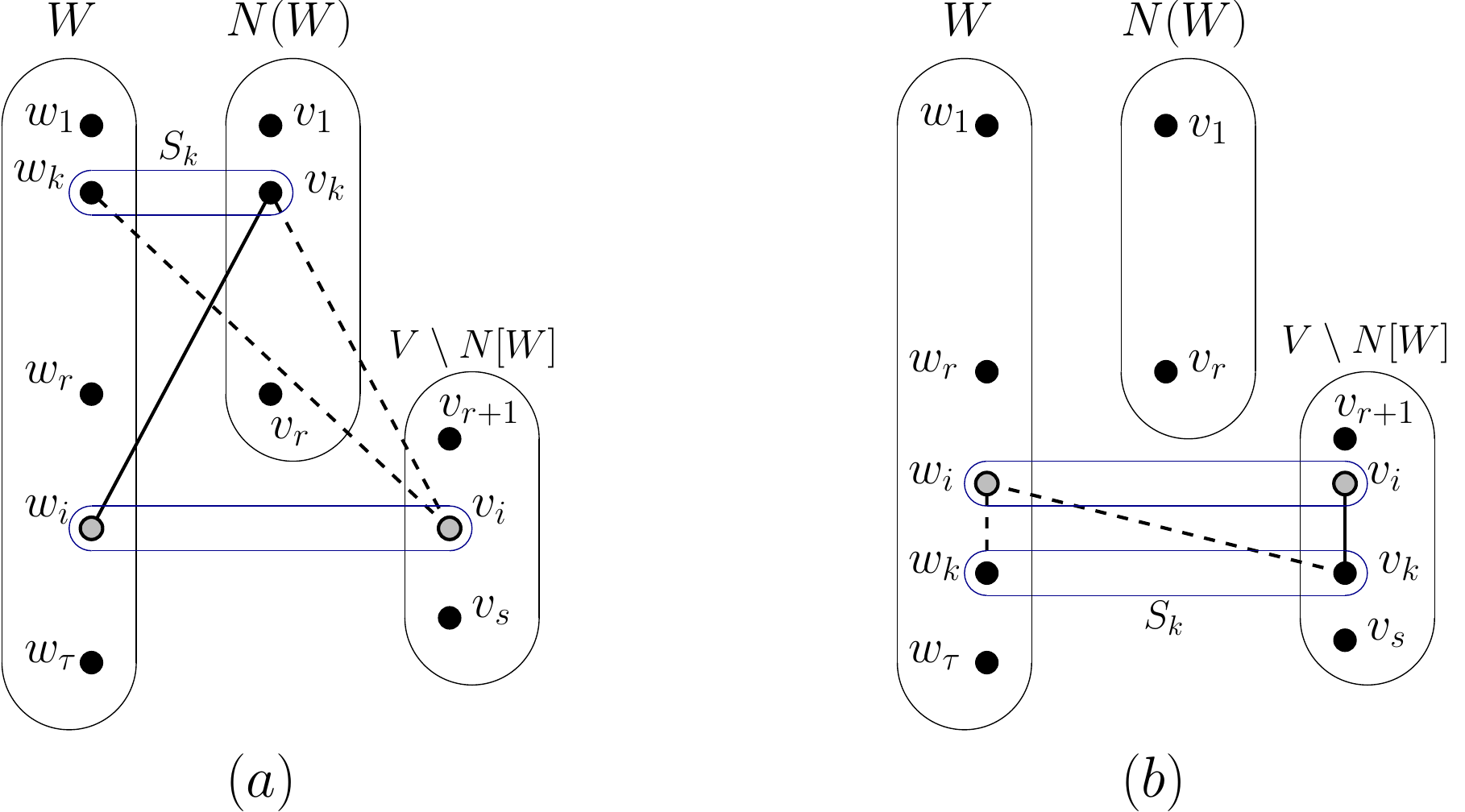}
\caption{In both cases, the part $S_k$ resolves the pair $w_i$, $v_i$.
Solid lines hold for adjacent vertices and dashed lines, for non-adjacent vertices. }\label{fig.twins}
\end{center}
\end{figure}
	
\textbf{Case 2.2}: Vertex  $v_i$ is adjacent to all vertices in $\{ v_1,\dots ,v_r\}$. As $v_i\notin W$, $v_iv_k\in E(G)$ for some $k\in \{r+1,\ldots, s\}$.
Consider the part $S_k=\{ w_k,v_k\}$. 
On the one hand,  $d(w_i,S_k)=2$ since $d(w_i,w_k)=2$ and $d(w_i,v_k)\ge 2$. 
On the other hand,  $d(v_i,S_k)=1$ since $d(v_i,v_k)=1$. 
Therefore, $d(w_i,S_k)\not= d(v_i,S_k)$  (see Figure \ref{fig.twins}(b)).
\end{proof}

As a direct consequence of Proposition \ref{cal}, Theorem \ref{kmedios} and  Proposition \ref{cal0}, the following result is derived.

\begin{thm} \label{txulisimo}
Let $G$ be a  graph of order $n$, other than $K_n$, such that $\tau(G)=\tau>\frac{n}{2}$. Then,  $\tau \le  \beta_p(G)  \le \frac{n+\tau}{2}$.
Moreover, if  $W$ is its $\tau$-set, then
\begin{enumerate}
\item $\beta_p(G)=\tau$  if and only if  $G[W]\cong\overline{K_{\tau}}$.
\item $\tau <  \beta_p(G)   \le \frac{n+\tau}{2}$ if and only  if   $G[W]\cong K_{\tau}$.
\end{enumerate}
\end{thm}

\begin{cor} \label{xulisimo}
Let $G$ be a nontrivial graph of order $n$, other than $K_n$, such that $\beta_p(G)=n-h$ and $\tau(G)=\tau>\frac{n}{2}$. 
Let  $W$ be its $\tau$-set. 
Then, $n-2h \le \tau \le n-h-1$ if and only  if $G[W]\cong K_{\tau}$.
\end{cor}

\begin{cor} \label{f1234}
For every $n\ge7$, the graphs $F_1$, $F_2$, $F_3$ and $F_4$, displayed in Figure \ref{pdn-2 wrong}, satisfy $\beta(F_i)=n-3$.
\end{cor}

\section{Partition dimension almost the order}

Our aim in this section is to completely characterize the set of all graphs of order $n\ge9$ such that $\beta_p(G)=n-2$.
This issue was already approached in \cite{tom08}, but, as remarked in our introductory section, the list of 23 graphs presented for every order $n\ge9$ turned out to be  wrong.

As was shown in Proposition \ref{twin.list}(8), it is clear that the only graphs whose partition dimension equals its order, are the complete graphs. The next result, along with Proposition \ref{twin.list}(9) and Proposition \ref{prop.twingran}, allows us to characterize, in a pretty simple way,  all connected graphs of order $n$ with partition dimension $n-1$, a result already proved in \cite{ChaSaZh00} for $n\ge3$.for the case $\beta_p(G)=n-1$,

\begin{prop}\label{n-1iff}
Let $G$ be a  graph of order $n\ge 9$ and twin number $\tau$, and let $W$ be a $\tau$-set. 
Then,  $\beta_p(G)=n-1$ if and only if $G$ satisfies one of the following conditions: 
\begin{enumerate}[(i)]
\item $\tau=n-1$.
\item $\tau=n-2$ and $G[W]\cong K_{n-2}$.
\end{enumerate}
\end{prop}
\begin{proof} 
Suppose that $\beta_p(G)=n-1$. 
Then, by Theorem \ref{thm.twinpetit}, $\tau>\frac{n}{2}$.
Thus, by  Theorem \ref{txulisimo} and Corollary \ref{xulisimo}, $n-2 \le \tau \le n-1$ and if $\tau=n-2$, then $G[W]\cong K_{n-2}$.

If $\tau=n-1$, i.e., if $G\cong K_{1,n-1}$, then $\beta_p(G)=n-1$. 
If $\tau=n-2$ and $G[W]\cong K_{n-2}$ then, according to Proposition \ref{cal}, $\beta_p(G)\ge \tau +1=n-1$, which means that $\beta_p(G)=n-1$, as $G$ is not the complete graph.
\end{proof}

\begin{cor} (\cite{ChaSaZh00})\label{pretty}
Let $G$ be a  graph of order $n\ge 9$.
Then, $\beta_p(G)=n-1$ if and only if $G$ is one of the following graphs:

\begin{enumerate}

   \item the star $K_{1,n-1}$.
	
   \item the complete split graph $K_{n-2} \vee \overline{K_2}$ obtained by removing an edge $e$ from the complete  graph $K_n$  (see Figure \ref{bpn-2}(a)).
	
   \item  the graph $K_1\vee (K_1+K_{n-2})$ obtained by attaching a leaf to the complete graph $K_{n-1}$  (see Figure \ref{bpn-2}(b)).
	
\end{enumerate}

\end{cor}

Next, we approach the case $\beta_p(G)=n-2$.

\begin{defi}
{\rm Let $G=(V,E)$ a graph such that $\tau(G)=\tau$.
Let $W$ be a $\tau$-set of $G$ such that $G[W]\cong K_{\tau}$.
A vertex $v\in V \setminus W$ is said to be a \emph{$W$-distinguishing vertex} of $G$ if and only if, for every vertex $z \in N(W) \setminus W$,
$d(v,z) \neq d(v,W)$.}
\end{defi}

\begin{lem}\label{lema3}
Let $G=(V,E)$ be a nontrivial graph of order $n$ such that $\tau(G)=\tau> \frac{n}{2}$.
Suppose that its $\tau$-set $W$ satisfies $G[W]\cong K_{\tau}$. Then, the following statements hold:
\begin{enumerate}

\item[(a)] If $G$ contains a $W$-distinguishing vertex, then $\beta_p(G)=\tau+1$.

\item[(b)] If $G[N(W)\setminus W]$ contains an isolated  vertex, then $\beta_p(G)=\tau+1$.

\item[(c)] If $|N(W) \setminus W|=1$, then $\beta_p(G)=\tau+1$.

\item[(d)]
If $G[N(W)\setminus W]$ contains a universal vertex $v$, then $v$ is adjacent to at least one vertex of $V\setminus N [W]$.

\end{enumerate}
\end{lem}
\begin{proof} 
\begin{enumerate}

\item[(a)]
Let $v$ be a $W$-distinguishing vertex. Set $W=\{w_1,w_2,\ldots,w_{\tau}\}$ and $V\setminus W=\{v, z_1,\ldots,z_{r}\}$, where $r=n-\tau-1<\tau$.
Take the partition
$\Pi=\{S_1,\ldots,S_{\tau+1}\}$, where:
$ S_1=\{w_1,z_1\}, \ldots, S_r=\{w_r,z_r\}, S_{r+1}=\{w_{r+1}\}, \ldots, S_{\tau}=\{w_{\tau}\}, S_{\tau+1}=\{v\}$.
Notice that if $z_i\in N(W)\setminus W$, then
$$d(z_i,S_{\tau +1})=d(z_i,v)\not= d(v,W)=d(v,w_i)=d(w_i,S_{\tau +1}),$$
and if $z_i\notin N(W)$, then for any $j\in \{1,\dots,\tau \}$ such that $i\not= j$ we have
$$d(z_i,S_{j})=d(z_i,w_j)>1= d(w_i,w_j)=d(w_i,S_{j}).$$
Thus, $r(w_i|\Pi)\neq r(z_i|\Pi)$ for every $i\in\{1,\ldots,r\}$, and consequently $\Pi $ is a locating partition of $G$.

\item[(b)] 
If $v$ is an isolated vertex in $G[N(W)\setminus W]$, then for every vertex $z \in N(W) \setminus W$,
$d(v,z)=2 \neq 1 = d(v,W)$.
Hence, $v$ is a $W$-distinguishing vertex of $G$,
and, according to item (a), $\beta_p(G)=\tau +1$.

\item[(c)] In this case,
the only vertex in $N(W)\setminus W$ is isolated in $G[N(W)\setminus W]$ and, according to item (b), $\beta_p(G)=\tau +1$.

\item[(d)]
Notice that if $v$ is universal in $G[N(W)\setminus W]$ and has no neighbor in $V\setminus N[W]$, then $v$ would be a twin of any vertex in $W$, which is a contradiction.
\end{enumerate}
\vspace{-.8cm}\end{proof}

As a straightforward consequence of item (b) of the previous lemma, the  following  holds.

\begin{cor} \label{f56}
For every $n\ge 9$, the graphs  $F_5$ and $F_6$, displayed in Figure \ref{pdn-2 wrong}, satisfy $\beta(F_i)=n-3$.
\end{cor}

\begin{lem}\label{lema4}
Let $G=(V,E)$ be a graph of order $n\ge9$ such that $\tau(G)=\tau=n-4$
and its $\tau$-set $W$ satisfies $G[W]\cong K_{\tau}$.
If $|N(W) \setminus W|=2$, then $\beta_p(G)=n-3$.
\end{lem}
\begin{proof} 
Let $W=\{w_1,\ldots,w_{n-4}\}$, $V\setminus W=\{z_1,z_2,z_3,z_4\}$ and $N(W)\setminus W=\{z_1,z_2\}$.

If $z_1z_2\not\in E$, then both $z_1$ and $z_2$ are isolated vertices in $G[N(W)\setminus W]$.
Hence, by Lemma \ref{lema3}(b),  $\beta_p(G)=\tau(G)+1=n-3$.

Suppose that $z_1z_2\in E$. 
According to Lemma \ref{lema3}(b), both $z_1$ and $z_2$ are adjacent to at least one vertex of  $V \setminus N(W)=\{z_3,z_4\}$. 
If, for some $i\in\{3,4\}$, either $\{z_1z_i,z_2z_i\}\subset E$, then $z_i$ is a $W$-distinguishing vertex, 
i.e., $\beta_p(G)=\tau+1=n-3$.

Assume thus that $\{z_1z_3,z_2z_4\}\subset E$ and $\{z_1z_4,z_2z_3\}\cap E =\emptyset$ (see Figure \ref{fig_lemas7y8}(a)).
Notice that none of the vertices of $V\setminus W$ is $W$-distinguishing.
Take the partition $\Pi=\{S_1,\ldots,S_{n-3}\}$, where 
$$ S_1=\{w_1,z_1\},  S_2=\{w_2,z_2\},  S_3=\{z_3,z_4\}, S_4=\{w_3\},\ldots, S_{n-3}=\{w_{n-4}\}.$$
Clearly, $\Pi$ is a locating partition of $G$, since
$d(w_1|S_3)=2 \not= 1=d(z_1,S_3)$,
  $d(w_2|S_3)=2 \not= 1=d(z_2|S_3)$ and
  $d(z_3|S_1)=1 \not= 2=d(z_4|S_1)$.
Finally, from Proposition \ref{cal}, we derive that $\beta_p(G)=n-3$.
\end{proof}

As a straightforward consequence of  this lemma, the  following  holds.

\begin{cor} \label{f78}
For every $n\ge 9$, the graphs  $F_7$ and $F_8$, displayed in Figure \ref{pdn-2 wrong}, satisfy $\beta(F_i)=n-3$.
\end{cor}

\begin{figure}[t]
\begin{center}
\includegraphics[width=0.8\textwidth]{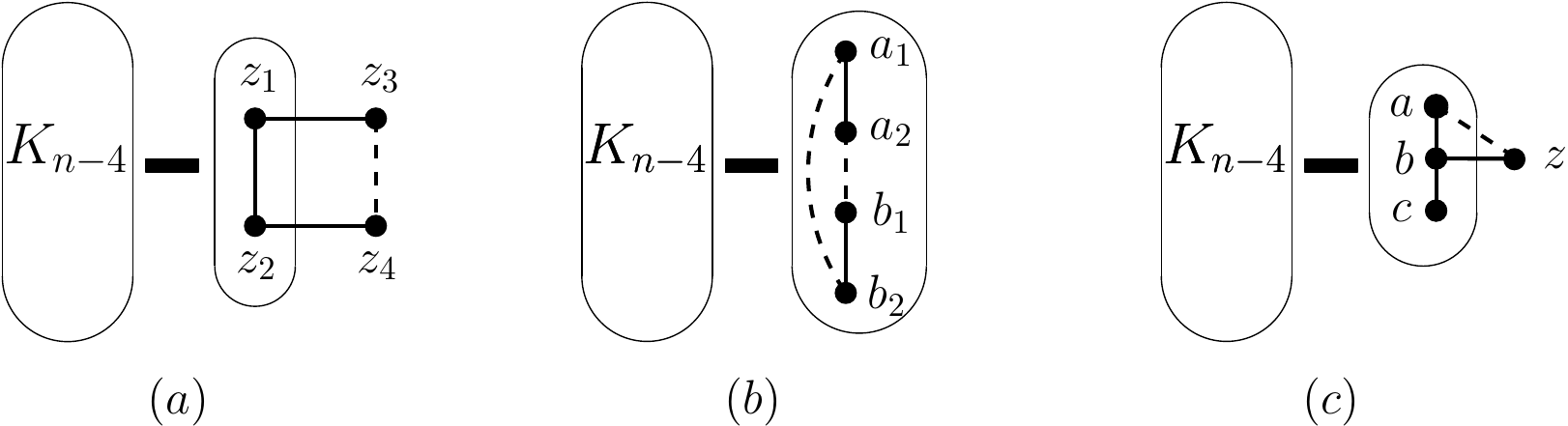}
\caption{In the three cases, $\tau=n-4$ and $G[W]\cong K_{n-4}$. 
Solid lines hold for adjacent vertices meanwhile dashed lines are optional.}
\label{fig_lemas7y8}
\end{center}
\end{figure}

\begin{figure}[!hbt]
\begin{center}
\includegraphics[width=0.99\textwidth]{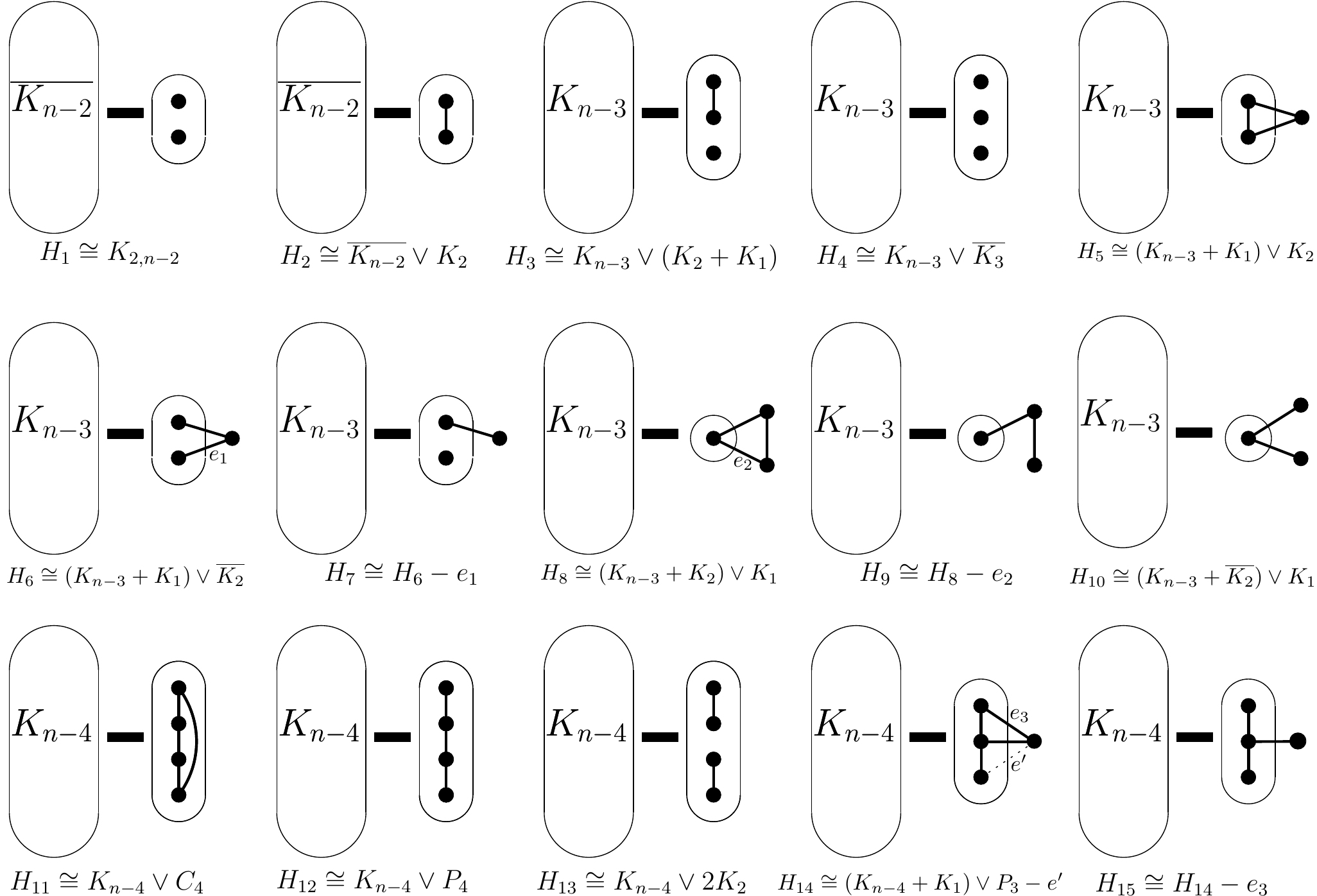}
\caption{These are all graph families such that $\beta_p(G)=n-2$. 
If $i\in\{1,2\}$, then $\tau(H_i)=n-2$. 
If $i\in\{3,\ldots,10\}$, then $\tau(H_i)=n-3$.
If $i\in\{11,\ldots,15\}$, then $\tau(H_i)=n-4$. }\label{taun234}
\end{center}
\end{figure}

\begin{thm} \label{n-2 true}
Let $G$ be a  graph of order $n\ge9$. 
Then,  $\beta_p(G)=n-2$ if and only if $G$ belongs to the following family $\{H_i\}_{i=1}^{15}$ (see Figure \ref{taun234}):

\begin{center}
\begin{tabular}{lll}
 $H_1 \cong  K_{2,n-2}$ &
 $H_2\cong \overline{K_{n-2}} \vee K_2$ &
 $H_3\cong K_{n-3}\vee (K_2+K_1)$    \\
 $H_4\cong K_{n-3}\vee \overline{K_{3}} $ &
 $H_5\cong (K_{n-3}+K_1)\vee K_2 $ &
 $H_6\cong  (K_{n-3}+K_1)\vee \overline{K_2} $    \\
$ H_7\cong  H_6 - e_1 $  &
$ H_8\cong  (K_{n-3} + K_2) \vee K_1 $ &
$ H_9\cong H_8 - e_2 $\\
$ H_{10}\cong  (K_{n-3}+\overline{K_2})\vee K_1 $&
$ H_{11}\cong  K_{n-4}\vee C_4  $  &
$ H_{12}\cong K_{n-4}\vee P_4  $  \\
$ H_{13}\cong  K_{n-4}\vee 2\, K_2  $ &
$ H_{14}\cong  (K_{n-4} + K_1) \vee P_3 - e'$ &
$ H_{15}\cong H_{14} - e_3 $ \\
\end{tabular}
\end{center}

\vspace{.2cm}\noindent
where
$e'$ is an edge joining the vertex of $K_1$ with an endpoint of  $P_3$ in $(K_{n-4} + K_1)$
and
$e_1$ is an edge joining the vertex of $K_1$ with a vertex of  $\overline {K_2}$ in $H_6$;
$e_2$ is an edge joining the vertex of $K_1$ with a vertex of  ${K_2}$ in $H_8$;
$e_3$ is an edge joining the vertex of $K_1$ with an endpoint of  ${P_3}$ in $H_{14}$.

\end{thm}
\begin{proof} ($\Longleftarrow$)
First suppose that $G$ is a graph belonging to the family $\{H_i\}_{i=1}^{15}$.
We distinguish three cases.

\noindent \textbf{Case 1}: $G\in\{H_1,H_2\}$. 
Hence, $\tau(G)=n-2$ and its $\tau$-set $W$ satisfies $G[W]\cong \overline{K_{n-2}}$.
Thus, according to Proposition \ref{cal0}, $\beta_p(G)=n-2$.

\noindent \textbf{Case 2}: $G\in\{H_i\}_{i=3}^{10}$. 
Hence, $\tau(G)=n-3$ and its $\tau$-set $W$ satisfies $G[W]\cong K_{n-3}$.
Thus,  according to Proposition \ref{cal},  $\beta_p (G) \ge \tau(G)+1=n-2$. 
Furthermore, from Proposition \ref{n-1iff} we deduce that $\beta_p (G)=n-2$.

\noindent \textbf{Case 3}: $G\in\{H_i\}_{i=11}^{15}$. 
Clearly, for all these graphs ${\rm diam}(G)=2$, $\tau(G)=n-4$ and its $\tau$-set $W$ satisfies $G[W]\cong K_{n-4}$.  
According to Proposition \ref{cal} and Theorem \ref{pretty}, $n-3 \le \beta_p(G) \le n-2$.
Suppose that there exists a locating partition  $\Pi=\{S_1, \dots, S_{n-3}\}$ of cardinality $n-3$.
If $W=\{w_1, \dots, w_{n-4}\}$, assume  that, for every $i\in\{1,\ldots,n-4\}$,  $w_i\in S_i$.
We distinguish two cases.

\textbf{Case 3.1}: 
 $G \in \{H_{11},H_{12},H_{13}\}$.
Note that $N(W)=V(G)$ and in all cases
there is a labelling $V(G) \setminus W=\{a_1,a_2,b_1,b_2\}$  
such that $d(a_1,a_2)=1$, $d(b_1,b_2)=1$, $d(a_1,b_1)=2$ and $d(a_2,b_2)=2$ (see Figure \ref{fig_lemas7y8}(b)).
	
Observe that  $\vert S_{n-3}\vert=1$, since $r(z,\Pi)=(1, \dots,1,0)$ for every $z\in\{a_1,a_2,b_1,b_2\}\cap S_{n-3}$. Notice also that
$\vert S_{i}\vert\leq2$ for  $i\in \{1, \dots ,n-4\}$, as for every $x\in S_i$, we have
$r(x,\Pi)=(1,\ldots,1,\overset{i)}{0},1,\ldots,1,h)$, with $h\in\{1,2\}$.
Hence, there are exactly three sets of $\Pi$ of cardinality 2.
We can suppose without loss of generality that $S_1=\{w_1,x\}$, $S_2=\{w_2,y\}$, $S_3=\{w_3,z\}$ and $S_{n-3}=\{t\}$, 
where $\{x,y,z,t\}=\{a_1,a_2,b_1,b_2\}$.
Hence, $d(t,x)=d(t,y)=d(t,z)=2$, a contradiction.
	
\textbf{Case 3.2}:
 $G \in \{H_{14},H_{15}\}$. 
	Note that $\vert N(W) \setminus W\vert =3$ and that there is a labelling $V(G)\setminus W=\{a,b,c,z\}$ such that $N(W) \setminus W=\{a,b,c\}$, 
$d(a,b)=d(b,c)=d(b,z)=1$, $d(c,a)=d(c,z)=2$ and $d(a,z)\in\{1,2\}$ (see Figure \ref{fig_lemas7y8}(c)).
	
Notice that
	$\vert S_{n-3}\vert\le 2$,
	since for every $x\in\{a,b,c\}\cap S_{n-3}$,  $r(x,\Pi)=(1, \dots,1,0)$.
	Moreover, $b\notin S_{n-3}$,
	otherwise $a$ and $c$ do not belong to $S_{n-3}$
	and we would have $r(a,\Pi)=r(c,\Pi)=(1,\ldots,1,1)$.
	So, we can assume without loss of generality that $\{w_1,b\}\subseteq S_1$.
	If $\{a,c\}\cap S_{n-3}\neq \emptyset$, then $r(w_1,\Pi)=r(b,\Pi)=(1,\ldots,1,1)$.
	Consequently,
	$r(w_i,\Pi)=r(c,\Pi)=(1,\ldots,1,2)$ for every $i\in \{1,\ldots,n-4\}$, a contradiction.

($\Longrightarrow$) Now assume that $G$ is a graph such that $\beta_p(G)=n-2$.
By Theorem \ref{thm.twinpetit}, $\tau(G)>\frac{n}{2}$, and according to Corollary \ref{xulisimo}, we have $n-4 \le \tau(G) \le n-2$.
We distinguish three cases, depending on the cardinality of $\tau(G)$.

\noindent \textbf{Case 1}: $\tau(G)=n-2$. 
Thus, according to Proposition \ref{prop.twingran} and Theorem \ref{pretty}, $G\in\{H_1,H_2\}$.

\noindent \textbf{Case 2}: $\tau(G)=n-3$. 
In thid case, from Proposition \ref{cal0} we deduce that its $\tau$-set $W$ satisfies $G[W]\cong K_{n-3}$.
We distinguish three cases, depending on the cardinality of $N(W)\setminus W$.

\textbf{Case 2.1}: $|N(W)\setminus W|=3$.
In this case, $G[N(W)\setminus W] \in \{K_3,  P_3,  K_2+K_1, \overline{K_3} \}$.
If $G[N(W)\setminus W]$ is  $K_3$ or $P_3$, then $\tau (G)\ge n-2$, a contradiction.
If $G[N(W)\setminus W]\cong K_2+K_1$, then $G\cong H_3$, and if
$G[N(W)\setminus W]\cong\overline{K_3}$, then $G\cong H_4$.

\textbf{Case 2.2}: 
$|N(W)\setminus W|=2$.
In this case, $G[N(W)\setminus W] \in \{K_2, \overline{K_2} \}$ and
 $|V\setminus N(W)|=1$.
 Let $z$ be the vertex in $V\setminus N(W)$.
 If $G[N(W)\setminus W]\cong K_2$ and $\deg(z)=1$, then $\tau(G)=n-2$, a contradiction.
 If $G[N(W)\setminus W]\cong K_2$ and $\deg(z)=2$, then $G\cong H_5$.
 If $G[N(W)\setminus W]\cong  \overline{K_2}$ and $\deg(z)=2$, then $G\cong H_6$.
 Finally, if $G[N(W)\setminus W]\cong  \overline{K_2}$ and $\deg(z)=1$, then $G\cong H_7$.

\textbf{Case 2.3}:
$|N(W)\setminus W|=1$.
Let $N(W)\setminus W=\{x\}$ and  $V\setminus N(W)=\{y,z\}$.
If $\deg(y)=\deg(z)=2$, then $G\cong H_8$.
If $\{\deg(y),\deg(z)\}=\{1,2\}$, then $G\cong H_{9}$.
If $\deg(y)=\deg(z)=1$, then $G\cong H_{10}$.

\noindent \textbf{Case 3}: $\tau(G)=n-4$. 
Let $W$ be its $\tau$-set.
In this case, from Proposition \ref{cal0} we deduce that its $\tau$-set $W$ satisfies $G[W]\cong K_{n-4}$.
Moreover, from Lemmas \ref{lema3} and \ref{lema4}, we deduce that
$G$ does not contain any $W$-distinguishing vertex
and $|N(W)\setminus W|\ge 3$.
Hence, $3\le |N(W)\setminus W|\le 4$.
We distinguish two cases, depending on the cardinality of $N(W)\setminus W$.

\textbf{Case 3.1}: $|N(W)\setminus W|=4$.
According to Lemma \ref{lema3}, all vertices of $G[N(W)\setminus W]$ have degree either 1 or 2.
Thus, $G[N(W)\setminus W]$ is a isomorphic to either $C_4$ or $P_4$ or  $2K_2$.
Hence, $G$ is isomorphic  to either $H_{11}$ or $H_{12}$ or $H_{13}$.

\textbf{Case 3.2}: $|N(W)\setminus W|=3$.
According to Lemma \ref{lema3}, $G[N(W)\setminus W]$ is a either $C_3$ or a $P_3$.
Suppose that $G[N(W)\setminus W]$ is $C_3$.
Then, by Lemma \ref{lema3}(d), every vertex of $N(W)\setminus W$ is adjacent to the unique vertex $z$ of $V\setminus N(W)$, a contradiction since in this case $z$ would be a $W$-distinguishing vertex.
Thus, $G[N(W)\setminus W]$ is  $P_3$.
According to Lemma \ref{lema3}(d), the central vertex $w$ of $P_3$
is adjacent to the unique vertex $z$ of $V\setminus N(W)$.
Observe also that one of the remaining two vertices of this path may be adjacent to vertex $z$, but not both,
since in this case $z$ would be a $W$-distinguishing vertex.
Hence, $G$ is isomorphic
to either $H_{14}$ or  $H_{15}$.
\end{proof}

\section*{Acknowledgements}
\vspace{-.3cm} Research partially supported by grants MINECO MTM2015-63791-R, Gen. Cat. DGR 2014SGR46 and MTM2014-60127-P.


\end{document}